\documentclass[11pt]{amsart}
\usepackage{hyperref}

\usepackage{graphicx}
\usepackage{amsmath,amssymb,mathrsfs}
\usepackage{amsthm}

\newtheorem{thm}{Theorem}[section]
\newtheorem*{thm*}{Theorem}
\newtheorem{cor}[thm]{Corollary}
\newtheorem{lem}[thm]{Lemma}

\theoremstyle{definition}
\newtheorem{defn}[thm]{Definition}
\theoremstyle{remark}
\newtheorem{rem}[thm]{Remark}
\theoremstyle{example}
\newtheorem{exa}[thm]{Example}
\theoremstyle{conjecture}

\numberwithin{equation}{section}

\newcommand{\calB}{\mathcal B}
\newcommand{\calE}{\mathcal E}
\newcommand{\calF}{\mathcal F}
\newcommand{\calH}{\mathcal H}
\newcommand{\calL}{\mathcal L}
\newcommand{\calM}{\mathcal M}
\newcommand{\calN}{\mathcal N}
\newcommand{\calP}{\mathcal P}
\newcommand{\calQ}{\mathcal Q}
\newcommand{\calS}{\mathcal S}
\newcommand{\calX}{\mathcal X}

\newcommand {\C} {\mathbb C}

\newcommand {\R} {\mathbb R}
\newcommand {\X} {\mathbb X}
\newcommand {\Z} {\mathbb Z}

\begin{document}

\title[Polynomial behavior in mean of stochastic skew-evolution semiflows]{Polynomial behavior in mean of stochastic skew-evolution semiflows}%

\date{\today}%

\author{Pham Viet Hai}%
\address[P. V. Hai]{ISE Department, National University of Singapore, 1 Engineering Drive 2, Singapore 117576, Singapore.}%
\email{phamviethai86@gmail.com}

\subjclass[2010]{93E15, 37L55}%

\keywords{polynomial stability, polynomial instability, skew-evolution semiflows, Datko's theorem, Banach function spaces}%


\maketitle

\begin{abstract}
In this paper, we are interested in the more general concept of a polynomial (in)stability in mean in which the polynomial behaviour in the classical sense is replaced by a weaker requirement with respect to some probability measure. This concept includes the classical concepts of a polynomial (in)stability as particular cases. Extending techniques employed in the deterministic case, we obtain variants of a well-known theorem of Datko for a polynomial (in)stability in mean. This is done by using the techniques of stochastic skew-evolution semiflows and Banach spaces of functions or sequences.
\end{abstract}

\section{Introduction}
A series of recent works have pointed out that an impressive list of classical problems and questions can be investigated employing the theory of skew-evolution (semi)flows. These (semi)flows arise naturally when one considers the linearization along an invariant manifold of a dynamical system generated by an autonomous differential equation. What makes skew-evolution (semi)flows important is the fact that they can be viewed as really generalizations of many well-known concepts in dynamical systems, such as $C_0$-(semi)groups, evolution families, and skew-product (semi)flows. 

The last decades have witnessed many momentous contributions in the study of the asymptotic behaviour of differential equations in abstract spaces. Many results can be carried out not only for differential equations and evolution families but also for skew-evolution semiflows. Among them, we can mention a famous result from the paper \cite{RD} of Datko: an exponentially bounded evolution family $\{U(t,s)\}_{t\geq s\geq 0}$ is exponentially stable on a Banach space $\X$ if and only if there exists $p>0$ such that
$$
\sup_{s\geq 0}\int_s^\infty\|U(t,s)x\|^p\,dt<\infty,\quad\forall x\in\X.
$$
The integral condition above means that for every $x\in\X$, the orbit $\tau\mapsto\|U(\tau+s,s)x\|$ belongs to the Lebesgue space $L^p(\R_{\geq 0})$ in a uniform way. Earlier, in \cite{DATKO1970610} Datko had obtained the exponent $p=2$ for strongly continuous semigroups. Since then Datko's theorem was the inspiration for a numerous number of works devoted to the existence of the exponential stability on the half-line. An interesting intervention on this subject is that of Zabczyk in \cite{zabczyk1974remarks}, where an analogous result for the discrete time was first obtained. Gradually, the techniques were improved and widened: from Datko-type characterization to Barbashin-type characterization (see \cite{PVH2, hai2011two, hai2012two}), from stability to instability (see \cite{MSS3}). Due its applications, Datko's theorem became one of the pillars of the modern control theory (see \cite{curtain2012introduction}). 

A notable improvement on this subject was given by  Neerven \cite{JMAMVN}, in which he discovered that the $p$-integrability of the associated orbits of a strongly continuous semigroup, which became so familiar in Datko's theorem, can be generalized to a more enhanced level, by replacing the Lebesgue space $L^p(\R_{\geq 0})$ with appropriate Banach function spaces. The paper \cite{JMAMVN} inspired the authors \cite{MSS1, MSS3} to characterize which skew-product semiflows are (un)stable in terms of the existence of various  function spaces. Some of results were generalized to the case of skew-evolution semiflows in \cite{PVH1, hai2011two, hai2012two}.

Over the last years, it can be seen an increasing interest in the research of a polynomial stability. Important contributions in the study of the existence of a polynomial stability for differential equations have been made and it is worth to mention here works \cite{barreira2009polynomial, bento2009stable}. These works have since stimulated intensive research on a polynomial stability. In \cite{PVH3}, the author presented a Datko-type characterization for a polynomially bounded evolution family to be polynomially stable. We recall a particular form from the paper \cite{PVH3}: a polynomially bounded evolution family $\{U(t,s)\}_{t\geq s\geq 0}$ is polynomially stable if and only if there exists $\delta>0$ such that
$$
\sup_{s\geq\delta}\int_s^\infty\dfrac{\|U(t,s)x\|dt}{t}<\infty,\quad\forall x\in\X.
$$
Doing the change of variables $t=\tau s$, the integral condition above is precisely the fact that for every $s\geq\delta$ and every $x\in\X$, the orbit $\tau\mapsto \|U(\tau s,s)x\|$ belongs to the weighted Lebesgue space $L_w^1(\R_{\geq 1})$ in a uniform way.

Naturally, the question arises whether Datko's theorem can be generalized to the case of a polynomial (in)stability in mean. The aim of this paper is to answer this question, and we even make a step beyond. We use the theory of Banach function spaces to characterize polynomially bounded stochastic skew-evolution semiflows, which are polynomially (un)stable in mean. It should be noted that the class of Banach function spaces used here is large enough to contain the weighted Lebesgue spaces $L_w^p(\R_{\geq 1})$, $p>0$ as very particular cases. Our characterizations are variants for the stochastic case of the famous theorem, of the deterministic case, due to Datko.

\section*{Notations}
Throughout the paper, we denote by $\Z$, $\R$, $\C$ by the sets of integers, real numbers, complex numbers, respectively. For $t\in\R$, the symbol $[t]$ stands for the greatest integer less than or equal to $t$. For a set $A\subseteq\R$, the symbol $\calX_A$ indicates the characteristic function of $A$, and $A_{\geq\delta}$ stands for the set $\{x\in A: x\geq\delta\}$. Denote $T=\{(t,s):t\geq s\geq 0\}$. We always denote by $\X$ a real or complex Banach space, and by $\calL(\X)$ the Banach algebra of all bounded linear operators on $\X$. The symbol $I$ stands for the identity operator on $\X$. The norm on $\X$ and on $\calL(\X)$ is denoted as $\|\cdot\|$. For $p>0$, we denote by $l^p_w(\Z_{\geq 1})$ the space of all sequences $s:\Z_{\geq 1}\to\R$ with $\sum_{j=1}^\infty \frac{|s(j)|^p}{j}<\infty$, and by $L_w^p(\R_{\geq 1})$ the space of all Lebesgue measurable functions $f:\R_{\geq 1}\to\R$ with $\int\limits_1^\infty\frac{|f(t)|^p}{t} dt<\infty$. For given constants $a,b>0$, we denote by $\calS(a,b)$ the set of non-decreasing sequence $\{t_n\}\subset\R_{\geq 1}$, with the following property
\begin{equation}\label{tab}
\dfrac{t_{mn}}{t_n}\leq am^b,\quad\forall m,n\in\Z_{\geq 1}.
\end{equation}

\section{Preliminaries}
\subsection{Skew-evolution semiflows}
Skew-evolution semiflows of the deterministic case were discussed in \cite{PVH1, SM} with motivations from differential equations and the study of Datko's theorem. In this section, we present a brief introduction to stochastic skew-evolution semiflows. We note the reader that stochastic cocycles studied in \cite{stoica2010uniform} are particular cases of the concepts below. We always denote by $(\calM,\calB,\mathbf{P})$ a probability space.
\begin{defn}
A measurable random field $\zeta:T\times\calM\to\calM$ is called a \emph{stochastic evolution semiflow} if
\begin{enumerate}
\item $\zeta(t,t,v)=v$, $\forall t\geq 0$, $\forall v\in\calM$.
\item $\zeta(t,s,v)=\zeta(t,r,\zeta(r,s,v))$, $\forall t\geq r\geq s\geq 0$, $\forall v\in\calM$.
\end{enumerate}
\end{defn}

\begin{exa}\label{example-varphi}
If $\varphi:\R_{\geq 0}\times\calM\to\calM$ is a stochastic semiflow (see \cite[Definition 2.1]{stoica2010uniform}), then the map $\zeta:T\times\calM\to\calM$ defined by $\zeta(t,s,v)=\varphi(t-s,v)$ is a stochastic evolution semiflow.
\end{exa}

\begin{exa}\label{example}
Let $\X$ be a real separable Hilbert space. $\calM$ is the space of all continuous paths $v:\R_{\geq 0}\to\X$ with $v(0)=0$, endowed with the compact open topology. Let $\calB_t$, where $t\geq 0$, be the $\sigma$-algebra generated by the set $\{v\to v(s)\in\X:s\leq t\}$ and let $\calB$ be the associated Borrel $\sigma$-algebra to $\calM$. Thus, for a Wiener measure $\mu$ on $\calM$, $(\calM,\calB,\calB_t,\mu)$ is a filtered probability space. Then $\zeta,\xi:T\times\calM\to\calM$ defined by
$$\zeta(t,s,v)(\tau)=\dfrac{t+1}{s+1}v(\tau),\quad\xi(t,s,v)(\tau)=v(\tau+t-s)-v(t-s)$$ are stochastic evolution semiflows.
\end{exa}

Several important works on the existence of stochastic semiflows for stochastic evolution equations emerged and the reader can refer to the monographs \cite{arnold1974stochastic, da2014stochastic}.

\begin{defn}
A map $\Phi:T\times\calM\to \calL(\X)$ is called a \emph{stochastic evolution cocycle} associated to an evolution semiflow $\zeta$ if
\begin{enumerate}
\item $\Phi(t,t,v)=I$, $\forall t\geq 0$, $\forall v\in\calM$.
\item $\Phi(t,s,v)=\Phi(t,r,\zeta(r,s,v))\Phi(r,s,v)$, $\forall t\geq r\geq s\geq 0$, $\forall v\in\calM$.
\end{enumerate}
In this case, the pair $(\Phi,\zeta)$ is called a \emph{stochastic skew-evolution semiflow}.
\end{defn}

We discuss some illustrative examples. Firstly, stochastic evolution cocycles describe solutions of variational equations and Cauchy problems.
\begin{exa}
Let $\zeta:T\times\calM\to\calM$ be a stochastic evolution semiflow and $A:\calM\to\calL(\X)$ be a continuous map. Consider the differential equation
$$
x'(t)=A(\zeta(t,s,v))x(t),\quad x(s)=h.
$$
If $x(\cdot)$ is the unique solution of the equation above with the initial condition $x(s)=h$, then the map $\Phi:T\times\calM\to \calL(\X)$ defined by $\Phi(t,s,v)h=x(t)$ is a stochastic evolution cocycle.
\end{exa}

Secondly, an evolution family can be viewed as an evolution cocycle. For this, we recall that a two-parameter family $\{U(t,s)\}_{t\geq s\geq 0}$ of bounded linear operators on $\X$ is called an \emph{evolution family} if it satisfies the following conditions: (i) $U(t,t)=I$, $\forall t\geq 0$. (ii) $U(t,s)=U(t,r)U(r,s)$, $\forall t\geq r\geq s\geq 0$. 
\begin{exa}
If $\{U(t,s)\}_{t\geq s\geq 0}$ is an evolution family, then for any stochastic evolution semiflow $\zeta$, the map $\Phi_U$ defined by
$$
\Phi_U(t,s,v)=U(t,s),\quad t\geq s\geq 0, v\in\calM
$$
is a stochastic evolution cocycle. Thus, for any stochastic evolution semiflow $\zeta$, the pair $(\Phi_U,\zeta)$ is a stochastic skew-evolution semiflow.
\end{exa}

Furthermore, the stochastic evolution cocycles are also generalizations of stochastic cocycles (see \cite{PVH1, PVH2}).
\begin{exa}
If $\phi:\R_{\geq 0}\times\calM\to\calL(\X)$ is a stochastic cocycle associated to the stochastic semiflow $\varphi:\R_{\geq 0}\times\calM\to\calM$ (see \cite[Definition 2.2]{stoica2010uniform}), then the map $\Phi$ defined by $\Phi(t,s,v)=\phi(t-s,v)$ is a stochastic evolution cocycle associated to the evolution semiflow $\zeta$ in Example \ref{example-varphi}.
\end{exa}

Finally, stochastic evolution cocycles arise from stochastic differential equations (see \cite{arnold1974stochastic, da2014stochastic}, also see \cite{stoica2010uniform}).

Some definitions of asymptotic properties in the classical sense are given in the following.
\begin{defn}
A stochastic skew-evolution semiflow $(\zeta,\Phi)$ is said to be
\begin{enumerate}
\item \emph{polynomially bounded} if there exist constants $M,\theta,\omega>0$ such that
\begin{equation}\label{poly-growth-usual}
\|\Phi(t,s,y)x\|\leq M\left(\dfrac{t}{s}\right)^\omega\|x\|,
\end{equation}
for all $t\geq s\geq\theta$ and all $(y,x)\in\calM\times\X$.
\item \emph{polynomially stable} if it satisfies \eqref{poly-growth-usual} with $\omega<0$.
\end{enumerate}
\end{defn}

Equipping the probability measure $\mathbf{P}$ on $\calM$ can offer the concepts of stochastic stability. In connection with this, we recall that $\calL^1(\calM,\mathbf{P})$ is the Banach space of all Bochner measurable functions $g:\calM\to\X$ such that
$$
\|g\|_1=\int_{\calM}\|g(y)\|\,d\mathbf{P}(y)<\infty.
$$
Two functions in $\calL^1(\calM,\mathbf{P})$ are identified if $\mathbf{P}$-almost everywhere.
\begin{defn}
A stochastic skew-evolution semiflow $(\zeta,\Phi)$ is said to be
\begin{enumerate}
\item \emph{polynomially bounded in mean} if there exist constants $M,\theta,\omega>0$ such that
\begin{equation}\label{poly-growth}
\int_{\calM}\|\Phi(t,s,y)g(y)\|\,d\mathbf{P}(y)\leq M\left(\dfrac{t}{s}\right)^\omega\int_{\calM}\|g(y)\|\,d\mathbf{P}(y)
\end{equation}
for all $t\geq s\geq\theta$ and all $g\in\calL^1(\calM,\mathbf{P})$.
\item \emph{polynomially stable in mean} if it satisfies \eqref{poly-growth} with $\omega<0$.
\item \emph{polynomially unstable in mean} if there exist constants $K,\delta,\alpha>0$ such that
$$
\int_{\calM}\|\Phi(t,s,y)g(y)\|\,d\mathbf{P}(y)\geq K\left(\dfrac{t}{s}\right)^\alpha\int_{\calM}\|g(y)\|\,d\mathbf{P}(y)
$$
for all $t\geq s\geq\delta$ and all $g\in\calL^1(\calM,\mathbf{P})$.
\item \emph{injective in the stochastic mean} if for every $g\in\calL^1(\calM,\mathbf{P})$, we have
\begin{eqnarray}\label{injective-defn}
\int_{\calM}\|\Phi(t,s,y)g(y)\|\,d\mathbf{P}(y)>0.
\end{eqnarray}
\end{enumerate}
\end{defn}

\begin{rem}
We note the reader that any polynomial stable stochastic skew-evolution semiflow admits a polynomial stability in mean, but the converse direction fails to hold. In order to describe examples of stochastic skew-evolution semiflows that are polynomial stable in mean but that are not polynomial stable in the sense \eqref{poly-growth-usual}, we consider a partition $\calM=\cup_{j=0}^\ell\calM_j$ of $\calM$ into at most countably many sets, where the number $\ell$ may be finite or infinite. Assume that for each $j$, there exist $\alpha_j,M_j,\theta_j$ such that
$$
\int_{\calM_j}\|\Phi(t,s,y)g(y)\|\,d\mathbf{P}(y)\leq M_j\left(\dfrac{t}{s}\right)^{-\alpha_j}\int_{\calM_j}\|g(y)\|\,d\mathbf{P}(y)
$$
for all $t\geq s\geq\theta_j$ and all $g\in\calL^1(\calM,\mathbf{P})$. Moreover, assume that
$$
\alpha_0=0,\quad\calM_0\ne\emptyset,\quad\mu(\calM_0)=0,\quad\alpha=\inf\{\alpha_j:j\in\Z_{\geq 1}\}>0.
$$
Under these assumptions, $(\zeta,\Phi)$ is polynomially stable in mean. Since $\alpha_0=0$, $\calM_0\ne\emptyset$, and $\mu(\calM_0)=0$, it is not polynomially stable in the sense \eqref{poly-growth-usual}.
\end{rem}

\begin{rem}
It is straightforward to prove that if a stochastic skew-evolution semiflow is exponentially stable in mean, then it must be polynomially stable in mean. The converse direction is not valid. To give an
example for this claim, let $\zeta:T\times\calM\to\calM$ be the stochastic evolution semiflow in Example \ref{example}. The map $\Phi:T\times\calM\to\calL(\X)$ defined by
$$
\Phi(t,s,v)x=\dfrac{s+1}{t+1}x
$$
is a stochastic evolution cocycle associated to the evolution semiflow $\zeta$. For any $g\in\calL^1(\calM,\mathbf{P})$, we have
\begin{eqnarray}
\int_{\calM}\|\Phi(t,s,y)g(y)\|\,d\mathbf{P}(y)
\label{asfdg}&=&\dfrac{s+1}{t+1}\int_{\calM}\|g(y)\|\,d\mathbf{P}(y)\\
\nonumber&\leq&\dfrac{2s}{t}\int_{\calM}\|g(y)\|\,d\mathbf{P}(y),\quad\forall t\geq s\geq 1,
\end{eqnarray}
which gives that $(\zeta,\Phi)$ is polynomially stable in mean. We prove by a contradiction that $(\zeta,\Phi)$ is not exponentially stable in mean. Assume that it is exponentially stable in mean. Then there exist positive constants $K,\alpha,\delta$ such that
$$
\int_{\calM}\|\Phi(t,s,y)g(y)\|\,d\mathbf{P}(y) \leq Ke^{-\alpha (t-s)}\int_{\calM}\|g(y)\|\,d\mathbf{P}(y),\quad\forall t\geq s\geq 0,
$$
which implies, by \eqref{asfdg}, that $e^{\alpha t}(t+1)^{-1}\leq Ke^{\alpha s}(s+1)^{-1}$. Letting $t\to\infty$ gives the contradiction.
\end{rem}

\subsection{Banach function spaces}\label{Banach}
We always denote by $(\Omega,\Sigma,\mu)$ a positive $\sigma$-finite measure space, and by $\calF(\mu)$ the linear space of $\mu$-measurable functions from $\Omega$ to $\C$. Two functions in $\calF(\mu)$ are identical if $\mu$-almost everywhere.
\begin{defn}
A function $\calN:\calF(\mu)\to [0,\infty]$ is called a \emph{Banach function norm} if it satisfies
\begin{enumerate}
\item $\calN(g)=0$ if and only if $g=0$ $\mu$-almost everywhere;
\item if $|g|\leq |h|$ $\mu$-almost everywhere, then $\calN(g)\leq\calN(h)$;
\item $\calN(g+h)\leq\calN(g)+\calN(h)$, $\forall g,h\in \calF(\mu)$;
\item $\calN(zg)=|z|\calN(g)$, $\forall z\in\C$, $\forall g\in \calF(\mu)$ with $\calN(g)<\infty$.
\end{enumerate}
\end{defn}

Given a Banach function norm $\calN$, we consider $\Delta=\{f\in\calF(\mu):\calN(f)<\infty\}$. It can be proved that $\Delta$ is a normed linear space with respect to the norm defined by $|f|_\Delta:=\calN(f)$. The space $(\Delta,|\cdot|_\Delta)$ is called a \emph{Banach function space} over $\Omega$ if it is complete. For more details about the theory of Banach function spaces, we refer the reader to the monograph \cite{MNP}. 

In this paper, we are interested in two cases of $\Omega$:

- For $(\Omega,\Sigma,\mu)=(\Z_{\geq 1},\calP(\Z_{\geq 1}),\mu_c)$, where $\mu_c$ is the counting measure, we denote by $\calH(\Z_{\geq 1})$ the class of all Banach sequence spaces $\calE$ with
\begin{equation}\label{condi-HN}
\lim\limits_{p\to\infty}\inf_{n\in\Z_{\geq 1}}|\calX_{\{n,\cdots,pn\}}|_{\calE}=\infty.
\end{equation}
Interestingly, this class contains the space $l^p_w(\Z_{\geq 1})$, for $p>0$.

- For $(\Omega,\Sigma,\mu)=(\R_{\geq 1},\sigma(\R_{\geq 1}),m)$, where $m$ is the Lebesgue measure, we denote by $\calH(\R_{\geq 1})$ the class of all Banach function spaces $\calQ$ with
$$
\lim\limits_{t\to\infty}\inf_{s\geq 1}|\calX_{[s,ts)}|_{\calQ}=\infty.
$$
The class $\calH(\R_{\geq 1})$ contains the space $L_w^p(\R_{\geq 1})$, for $p>0$.

\begin{rem}\label{rem-import-banach-func}
If $A\in\calH(\R_{\geq 1})$, then the Banach sequence space
$$
S_A=\left\{\{\alpha_n\}_{n\geq 1}:\sum_{n=1}^\infty\alpha_n\calX_{[n,n+1)}\in A\right\}
$$
belongs to the class $\calH(\Z_{\geq 1})$ with respect to the norm
$$
|\{\alpha_n\}_{n\in\Z_{\geq 1}}|_{S_A}:=|\sum_{n=1}^\infty\alpha_n\calX_{[n,n+1)}|_A.
$$
\end{rem}

\section{Polynomial stability}
\subsection{Some initial properties}
In this subsection, we provide several facts that used to prove the main results. The following lemma offers a necessary and sufficient condition for a stochastic evolution cocycle to be polynomially stable in mean. It turns out that a polynomial stability in mean is related intimately to a contraction in the stochastic sense.
\begin{lem}\label{lemma-important}
Let $(\Phi,\zeta)$ be a stochastic skew-evolution semiflow, which is polynomially bounded in mean (that is \eqref{poly-growth} holds). Then it is polynomially stable in mean if and only if there exist $c\in (0,1)$, $\lambda\in\Z_{\geq 1}$, and $\delta>0$ such that the inequality
$$
\int_{\calM}\|\Phi(\lambda m,m,y)g(y)\|\,d\mathbf{P}(y)\leq c\int_{\calM}\|g(y)\|\,d\mathbf{P}(y)
$$
holds for every $m\in\Z_{\geq\delta}$ and every $g\in\calL^1(\calM,\mathbf{P})$.
\end{lem}
\begin{proof}
The necessity is clear. Let us prove the sufficiency. Take arbitrarily $g\in\calL^1(\calM,\mathbf{P})$. We can prove by induction on $k$ that
$$
\int_{\calM}\|\Phi(\lambda^k m,m,y)g(y)\|\,d\mathbf{P}(y)\leq c^k\int_{\calM}\|g(y)\|\,d\mathbf{P}(y),\quad\forall m\in\Z_{\geq\delta},\forall k\in\Z.
$$
Let $t\geq 1$, $m\in\Z$ with $m\geq\gamma:=\max\{\delta,\theta+1\}$. Setting
$$
\alpha=-\frac{\ln c}{\ln\lambda}, \,K_1=\frac{M\lambda^\omega}{c},\,K=\max\left\{M\left(1+\dfrac{1}{[\gamma]}\right)^{\omega+\alpha},K_1M\left(1+\dfrac{1}{[\gamma]}\right)^{\omega+\alpha}\right\}.
$$
Let $p=\max\left\{j\in\Z:t\lambda^{-j}\geq 1 \right\}$. Then we have $\lambda^p\leq t<\lambda^{p+1}$. By \eqref{poly-growth}, we estimate
\begin{eqnarray*}
\int_{\calM}\|\Phi(tm,m,y)g(y)\|\,d\mathbf{P}(y)
&\leq& M\left(\dfrac{t}{\lambda^p}\right)^{\omega}\int_{\calM}\|\Phi(\lambda^p m,m,y)\|\,d\mathbf{P}(y)\\
&\leq&M\lambda^{\omega}c^p\int_{\calM}\|g(y)\|\,d\mathbf{P}(y),
\end{eqnarray*}
which implies, as $p>\frac{\ln t}{\ln\lambda}-1$, that
\begin{eqnarray}
\int_{\calM}\|\Phi(tm,m,y)g(y)\|\,d\mathbf{P}(y)
\nonumber &\leq& M\lambda^\omega c^{\frac{\ln t}{\ln\lambda}-1}\int_{\calM}\|g(y)\|\,d\mathbf{P}(y)\\
\label{cvb}&=& K_1 t^{-\alpha}\int_{\calM}\|g(y)\|\,d\mathbf{P}(y).
\end{eqnarray}
Let $r\geq s\geq\gamma$. For $r$, there are the following possibilities.

- If $r\leq [s]+1$, then by \eqref{poly-growth}, we have
\begin{eqnarray*}
\int_{\calM}\|\Phi(r,s,y)g(y)\|\,d\mathbf{P}(y)
&\leq& M\left(1+\dfrac{1}{[\gamma]}\right)^\omega\int_{\calM}\|g(y)\|\,d\mathbf{P}(y)\\
&\leq& M\left(1+\dfrac{1}{[\gamma]}\right)^{\omega+\alpha}\left(\dfrac{r}{s}\right)^{-\alpha}\int_{\calM}\|g(y)\|\,d\mathbf{P}(y)\\
&\leq& K\left(\dfrac{r}{s}\right)^{-\alpha}\int_{\calM}\|g(y)\|\,d\mathbf{P}(y).
\end{eqnarray*}

- If $r\geq [s]+1$, then by \eqref{cvb} and \eqref{poly-growth}, we estimate
\begin{eqnarray*}
&&\int_{\calM}\|\Phi(r,s,y)g(y)\|\,d\mathbf{P}(y)\\
&&\leq K_1\left(\dfrac{r}{[s]+1}\right)^{-\alpha}\int_{\calM}\|\Phi([s]+1,s,y)g(y)\|\,d\mathbf{P}(y)\\
&&\leq K_1\left(\dfrac{r}{[s]+1}\right)^{-\alpha}M\left(\dfrac{[s]+1}{s}\right)^\omega\int_{\calM}\|g(y)\|\,d\mathbf{P}(y)\\
&&\leq K\left(\dfrac{r}{s}\right)^{-\alpha}\int_{\calM}\|g(y)\|\,d\mathbf{P}(y),
\end{eqnarray*}
where the last inequality uses the followings
$$
([s]+1)^\alpha\leq (s+1)^\alpha\leq s^\alpha\left(1+\dfrac{1}{[\gamma]}\right)^\alpha
$$
and
$$
\left(\dfrac{[s]+1}{s}\right)^\omega\leq\left(1+\dfrac{1}{[s]}\right)^\omega\leq\left(1+\dfrac{1}{[\gamma]}\right)^\omega.
$$
\end{proof}

The following lemma studies a uniform boundedness in mean of an evolution cocycle.
\begin{lem}\label{t-bdd}
Let $(\Phi,\zeta)$ be a stochastic skew-evolution semiflow, which is polynomially bounded in mean (that is \eqref{poly-growth} holds). Assume that there exist constants $\delta,L>0$, an above unbounded sequence $\{t_n\}\in\calS(a,b)$ such that the inequality
$$
\int_{\calM}\|\Phi(t_nm,m,y)g(y)\|\,d\mathbf{P}(y)\leq L\int_{\calM}\|g(y)\|\,d\mathbf{P}(y)
$$
holds for every $n\in\Z_{\geq 1}$, $m\in\Z_{\geq\delta}$, and every $g\in\calL^1(\calM,\mathbf{P})$. Then there exists a constant $K>0$ such that the inequality
$$
\int_{\calM}\|\Phi(ts,s,y)g(y)\|\,d\mathbf{P}(y)\leq K\int_{\calM}\|g(y)\|\,d\mathbf{P}(y)
$$
holds for every $t\geq 1$, $s\geq\max\{\delta,\theta+1\}$, and every $g\in\calL^1(\calM,\mathbf{P})$.
\end{lem}
\begin{proof}
Let $t\geq 1$, $g\in\calL^1(\calM,\mathbf{P})$, and $m\in\Z_{\geq\gamma}$, where $\gamma:=\max\{\delta,\theta+1\}$. Since $\lim\limits_{n\to\infty}t_n=\infty$, we can find $p\in\Z$ satisfying
$$
t_{2d}\geq t,\quad\forall d\geq p.
$$
Setting $j=\min\{k\in\{1,\cdots,p\}:t\leq t_{2k}\}$. There are two possibilities of $j$.

\noindent - If $j=1$, then $t\leq t_2$, and so by \eqref{poly-growth}
$$
\int_{\calM}\|\Phi(tm,m,y)g(y)\|\,d\mathbf{P}(y)\leq Mt_2^\omega\int_{\calM}\|g(y)\|\,d\mathbf{P}(y).
$$

\noindent - If $j\geq 2$, then $t\leq t_{2j}$ and $t>t_{2(j-1)}\geq t_j$. By \eqref{poly-growth}, we have
\begin{eqnarray*}
\int_{\calM}\|\Phi(tm,m,y)g(y)\|\,d\mathbf{P}(y)
&\leq& M\left(\dfrac{t}{t_j}\right)^\omega\int_{\calM}\|\Phi(t_j m,m,y)g(y)\|\,d\mathbf{P}(y)\\
&\leq& ML\left(\dfrac{t}{t_j}\right)^\omega\int_{\calM}\|g(y)\|\,d\mathbf{P}(y)\\
&\leq& ML2^{\omega b}a^\omega\int_{\calM}\|g(y)\|\,d\mathbf{P}(y),
\end{eqnarray*}
where the last inequality holds by condition \eqref{tab}, and $t\leq t_{2j}$. Thus, we get
\begin{eqnarray}\label{kjl}
\int_{\calM}\|\Phi(tm,m,y)g(y)\|\,d\mathbf{P}(y)\leq K_1\int_{\calM}\|g(y)\|\,d\mathbf{P}(y),
\end{eqnarray}
where $K_1=\max\left\{Mt_2^\omega,ML2^{\omega b}a^\omega\right\}$.

Let $s\geq\gamma$. We have two cases of $s$ as follows.

\noindent - If $ts\leq [s]+1$, then by \eqref{poly-growth} we see
\begin{eqnarray*}
\int_{\calM}\|\Phi(ts,s,y)g(y)\|\,d\mathbf{P}(y)
&\leq& M\left(1+\dfrac{1}{[\gamma]}\right)^\omega\int_{\calM}\|g(y)\|\,d\mathbf{P}(y).
\end{eqnarray*}

\noindent - If $ts\geq [s]+1$, then by \eqref{kjl} and \eqref{poly-growth} we estimate
\begin{eqnarray*}
\int_{\calM}\|\Phi(ts,s,y)g(y)\|\,d\mathbf{P}(y)
&\leq& K_1\int_{\calM}\|\Phi([s]+1,s,y)g(y)\|\,d\mathbf{P}(y)\\
&\leq& K_1M\left(\dfrac{[s]+1}{s}\right)^\omega\int_{\calM}\|g(y)\|\,d\mathbf{P}(y).
\end{eqnarray*}
By choosing
$$
K:=\max\left\{M\left(1+\dfrac{1}{[\gamma]}\right)^\omega, K_1M\left(1+\dfrac{1}{[\gamma]}\right)^\omega\right\},
$$
we obtain the desired conclusion.
\end{proof}

\subsection{Discrete-time version}
With all preparation in place, we now are ready to state and prove the first main result, which is a discrete-time version of the Datko-type theorem.
\begin{thm}\label{datko-1}
Let $(\Phi,\zeta)$ be a stochastic skew-evolution semiflow, which is polynomially bounded in mean (that is \eqref{poly-growth} holds). Then it is polynomially stable in mean if and only if there exist $\delta>0$, a Banach sequence space $\calE\in\calH(\Z_{\geq 1})$, and $\{t_n\}\in\calS(a,b)$ such that
\begin{enumerate}
\item for every $(s,g)\in\R_{\geq\delta}\times\calL^1(\calM,\mathbf{P})$, the sequence $\psi_{s,g}:\Z_{\geq 1}\to\R_{\geq 0}$
$$
\psi_{s,g}(j)=\int_{\calM}\|\Phi(t_j s,s,y)g(y)\|\,d\mathbf{P}(y)
$$
belongs to $\calE$.
\item there exists $K>0$ such that
$$|\psi_{s,g}(\cdot)|_{\calE}\leq K\int_{\calM}\|g(y)\|\,d\mathbf{P}(y),\quad\forall s\in\R_{\geq\delta},\forall g\in\calL^1(\calM,\mathbf{P}).$$
\end{enumerate}
\end{thm}
\begin{proof}
$\bullet$ Necessity. It is immediate by taking $\calE=l^1_w(\Z_{\geq 1})$ and $t_n=n$.

$\bullet$ Sufficiency. Let $\gamma :=\max\{\delta,\theta+1\}$, $g\in\calL^1(\calM,\mathbf{P})$, and $n\in\Z_{\geq\gamma}$. There are two cases of $\{t_n\}$.

{\bf Case 1:} The sequence $\{t_n\}$ is above bounded. In this case, we can suppose that $p:=\sup\limits\{t_n:n\in\Z_{\geq 1}\}<\infty$.

Let $k\in\Z_{\geq 1}$. For every $j\in\{1,\cdots,k\}$, by \eqref{poly-growth} we have
\begin{eqnarray*}
\int_{\calM}\|\Phi(ps,s,y)g(y)\|\,d\mathbf{P}(y)
&\leq& M\left(\dfrac{p}{t_1}\right)^\omega\int_{\calM}\|\Phi(t_js,s,y)g(y)\|\,d\mathbf{P}(y),
\end{eqnarray*}
which gives
\begin{eqnarray*}
\psi_{s,g}(\cdot)\geq M^{-1}\left(\dfrac{p}{t_1}\right)^{-\omega}\int_{\calM}\|\Phi(ps,s,y)g(y)\|\,d\mathbf{P}(y)\calX_{\{1,\cdots,k\}},
\end{eqnarray*}
and hence, we can estimate
\begin{eqnarray*}
&& K\int_{\calM}\|g(y)\|\,d\mathbf{P}(y)\geq |\psi_{s,g}(\cdot)|_{\calE}\\
&&\geq M^{-1}\left(\dfrac{p}{t_1}\right)^{-\omega}\int_{\calM}\|\Phi(ps,s,y)g(y)\|\,d\mathbf{P}(y)\cdot|\calX_{\{1,\cdots,k\}}|_{\calE}.
\end{eqnarray*}
By condition \eqref{condi-HN}, we must have
$$
\int_{\calM}\|\Phi(ps,s,y)g(y)\|\,d\mathbf{P}(y)=0,\quad\forall s\in\R_{\geq\gamma},\forall g\in\calL^1(\calM,\mathbf{P}),
$$
and so, by Lemma \ref{lemma-important}, $(\Phi,\zeta)$ is polynomially stable in mean.

{\bf Case 2:} The sequence $\{t_n\}_{n\in\Z}$ is above unbounded.

Also by condition \eqref{condi-HN}, we can find $p$ with
$$
|\calX_{\{m,\cdots,mp\}}|_{\calE}\geq 1,\quad\forall m\in\Z_{\geq 1}.
$$
Let $k\in\Z_{\geq 1}$. Then there are $m\in\Z$, $r\in [0,p)$ such that $k=mp+r$. We consider two possibilities of $m$ as follows.

\noindent - If $m=0$, then by \eqref{poly-growth} we estimate
\begin{eqnarray*}
\int_{\calM}\|\Phi(t_k s,s,y)g(y)\|\,d\mathbf{P}(y)
&\leq& Mt_p^\omega\int_{\calM}\|g(y)\|\,d\mathbf{P}(y).
\end{eqnarray*}

\noindent - If $m\geq 1$, then for every $j\in\{m,\cdots,mp\}$, by \eqref{poly-growth} we have
\begin{eqnarray*}
\int_{\calM}\|\Phi(t_k s,s,y)g(y)\|\,d\mathbf{P}(y)
&\leq& M\left(\dfrac{t_k}{t_j}\right)^\omega\int_{\calM}\|\Phi(t_j s,s,y)g(y)\|\,d\mathbf{P}(y)\\
&\leq& M\left(\dfrac{t_{2mp}}{t_m}\right)^\omega\int_{\calM}\|\Phi(t_j s,s,y)g(y)\|\,d\mathbf{P}(y)\\
&\leq& Ma^\omega 2^{\omega b} p^{\omega b}\int_{\calM}\|\Phi(t_j s,s,y)g(y)\|\,d\mathbf{P}(y).
\end{eqnarray*}
The last inequality can be rewritten as
$$
\psi_{s,g}(\cdot)\geq M^{-1}a^{-\omega} 2^{-\omega b} p^{-\omega b}\int_{\calM}\|\Phi(t_k s,s,y)g(y)\|\,d\mathbf{P}(y)\calX_{\{m,\cdots,mp\}},
$$
and so we can estimate
\begin{eqnarray*}
&& K\int_{\calM}\|g(y)\|\,d\mathbf{P}(y)\geq |\psi_{s,g}(\cdot)|_{\calE}\\
&&\geq M^{-1}a^{-\omega} 2^{-\omega b} p^{-\omega b}\int_{\calM}\|\Phi(t_k s,s,y)g(y)\|\,d\mathbf{P}(y)\cdot|\calX_{\{m,\cdots,mp\}}|_{\calE}\\
&&\geq M^{-1}a^{-\omega} 2^{-\omega b} p^{-\omega b}\int_{\calM}\|\Phi(t_k s,s,y)g(y)\|\,d\mathbf{P}(y).
\end{eqnarray*}
Thus, both cases unveil that
$$
\int_{\calM}\|\Phi(t_k s,s,y)g(y)\|\,d\mathbf{P}(y)\leq \max\left\{Mt_p^\omega,KMa^{\omega} 2^{\omega b} p^{\omega b}\right\}\int_{\calM}\|g(y)\|\,d\mathbf{P}(y),
$$
and hence, by Lemma \ref{t-bdd}, there exists a constant $L>0$ such that
$$
\int_{\calM}\|\Phi(ts,s,y)g(y)\|\,d\mathbf{P}(y)\leq L\int_{\calM}\|g(y)\|\,d\mathbf{P}(y),\quad\forall t\geq 1,\forall s\geq\gamma.
$$

Let $\ell\in\Z_{\geq 1}$. For every $j\in\{1,\cdots,\ell\}$, we can estimate
$$
\int_{\calM}\|\Phi(t_{\ell} s,s,y)g(y)\|\,d\mathbf{P}(y)\leq L\int_{\calM}\|\Phi(t_j s,s,y)g(y)\|\,d\mathbf{P}(y),
$$
which is equivalent to
$$
\psi_{s,g}(\cdot)\geq L^{-1}\int_{\calM}\|\Phi(t_{\ell} s,s,y)g(y)\|\,d\mathbf{P}(y)\calX_{\{1,\cdots,\ell\}}.
$$
By the assumption, we estimate
\begin{eqnarray*}
K\int_{\calM}\|g(y)\|\,d\mathbf{P}(y)
&\geq& |\psi_{s,g}(\cdot)|_{\calE}\geq L^{-1}\int_{\calM}\|\Phi(t_{\ell} s,s,y)g(y)\|\,d\mathbf{P}(y)\cdot |\calX_{\{1,\cdots,\ell\}}|_{\calE}.
\end{eqnarray*}
On the other hand, by \eqref{poly-growth},
\begin{eqnarray*}
&&\int_{\calM}\|\Phi(([t_\ell]+1)s,s,y)g(y)\|\,d\mathbf{P}(y)\\
&&\leq M\left(\dfrac{[t_\ell]+1}{t_\ell}\right)^\omega\int_{\calM}\|\Phi(t_{\ell} s,s,y)g(y)\|\,d\mathbf{P}(y)\\
&&\leq M2^\omega\int_{\calM}\|\Phi(t_{\ell} s,s,y)g(y)\|\,d\mathbf{P}(y).
\end{eqnarray*}
Thus, we have
$$
\int_{\calM}\|\Phi(([t_\ell]+1)s,s,y)g(y)\|\,d\mathbf{P}(y)\cdot |\calX_{\{1,\cdots,\ell\}}|_{\calE}\leq KLM2^{\omega}\int_{\calM}\|g(y)\|\,d\mathbf{P}(y).
$$
By \eqref{condi-HN} we can choose $\ell$ with $|\calX_{\{1,\cdots,\ell\}}|_{\calE} \geq 2KLM2^{\omega}$, and so by Lemma \ref{lemma-important} we get the desired result.
\end{proof}

The following result is a direct consequence of the theorem above.
\begin{cor}
Let $(\Phi,\zeta)$ be a stochastic skew-evolution semiflow, which is polynomially bounded in mean (that is \eqref{poly-growth} holds). Then it is polynomially stable in mean if and only if there exist positive constants $K,\delta$ such that the inequality
$$
\sum_{j=1}^\infty\dfrac{1}{j}\int_{\calM}\|\Phi(js,s,y)g(y)\|\,d\mathbf{P}(y)\leq K\int_{\calM}\|g(y)\|\,d\mathbf{P}(y)
$$
holds for every $s\in\R_{\geq\delta}$ and every $g\in\calL^1(\calM,\mathbf{P})$.
\end{cor}

\subsection{Continuous-time version}
In this subsection, we give a continuous-time version of the Datko-type theorem by making use of Theorem \ref{datko-1}.
\begin{thm}\label{datko-2}
Let $(\Phi,\zeta)$ be a stochastic skew-evolution semiflow, which is polynomially bounded in mean (that is \eqref{poly-growth} holds). Then it is polynomially stable in mean if and only if there exist $\delta>0$, $\calQ\in\calH(\R_{\geq 1})$ such that
\begin{enumerate}
\item for every $(s,g)\in\R_{\geq\delta}\times\calL^1(\calM,\mathbf{P})$, the function $f_{s,g}:\R_{\geq 1}\to\R_{\geq 0}$
\begin{equation*}
f_{s,g}(t)=\int_{\calM}\|\Phi(ts,s,y)g(y)\|\,d\mathbf{P}(y)
\end{equation*}
belongs to $\calQ$;
\item there exists $K>0$ such that
$$|f_{s,g}(\cdot)|_{\calQ}\leq K\int_{\calM}\|g(y)\|\,d\mathbf{P}(y),\quad\forall s\in\R_{\geq\delta},\forall g\in\calL^1(\calM,\mathbf{P}).$$
\end{enumerate}
\end{thm}
\begin{proof}
$\bullet$ Necessity. It is immediate by taking $\calQ=L_w^1(\R_{\geq 1})$.

$\bullet$ Sufficiency. Let $S_{\calQ}$ be the Banach sequence space associated to $\calQ$, as mentioned in Remark \ref{rem-import-banach-func}. For every $(s,g)\in\R_{\geq\delta}\times\calL^1(\calM,\mathbf{P})$, let us define the sequence $\psi_{s,g}:\Z_{\geq 1}\to\R_{\geq 0}$ by setting
$$\psi_{s,g}(j)=\int_{\calM}\|\Phi(t_j s,s,y)g(y)\|\,d\mathbf{P}(y),$$
where $t_j=j+1$.

Let $d\geq 1$. Then there exists $k\in\Z_{\geq 1}$ such that $d\in [k,k+1)$, and hence, $[d]=k$. Thus, by \eqref{poly-growth}, we have
\begin{eqnarray*}
\int_{\calM}\|\Phi(t_k s,s,y)g(y)\|\,d\mathbf{P}(y)
&\leq& M\left(\dfrac{k+1}{d}\right)^\omega \int_{\calM}\|\Phi(ds,s,y)g(y)\|\,d\mathbf{P}(y)\\
&\leq& M2^\omega \int_{\calM}\|\Phi(ds,s,y)g(y)\|\,d\mathbf{P}(y),
\end{eqnarray*}
which gives $M 2^{\omega}f_{s,g}\geq \sum\limits_{k\geq 1}\psi_{s,g}(k)\calX_{[k,k+1)}$, and hence Remark \ref{rem-import-banach-func} shows $\psi_{s,g}\in S_{\calQ}$. By Theorem \ref{datko-1}, we conclude that $(\Phi,\zeta)$ is polynomially stable.
\end{proof}

As a consequence, we obtain the following result.
\begin{cor}
Let $(\Phi,\zeta)$ be a stochastic skew-evolution semiflow, which is polynomially bounded in mean (that is \eqref{poly-growth} holds). Then it is polynomially stable in mean if and only if there exist positive constants $K,\delta$ such that the inequality
$$
\int\limits_1^\infty\dfrac{1}{t}\int_{\calM}\|\Phi(ts,s,y)g(y)\|\,d\mathbf{P}(y)\,dt\leq K\int_{\calM}\|g(y)\|\,d\mathbf{P}(y)
$$
holds for every $s\in\R_{\geq\delta}$ and every $g\in\calL^1(\calM,\mathbf{P})$.
\end{cor}

\section{Polynomial instability}
\subsection{Some initial properties}
This section is devoted to proving auxiliary lemmas used later on. The following lemma offers a necessary and sufficient condition for a stochastic skew-evolution semiflow to be polynomially unstable in mean.
\begin{lem}\label{lemma-important}
Let $(\Phi,\zeta)$ be a stochastic skew-evolution semiflow, which is polynomially bounded in mean (that is \eqref{poly-growth} holds). Then it is polynomially unstable in mean if and only if there exist $c>1$, $\lambda\in\Z_{\geq 1}$, and $\delta>0$ such that the inequality
$$
\int_{\calM}\|\Phi(\lambda s,s,y)g(y)\|\,d\mathbf{P}(y)\geq c\int_{\calM}\|g(y)\|\,d\mathbf{P}(y)
$$
holds for every $s\in\R_{\geq\delta}$ and every $g\in\calL^1(\calM,\mathbf{P})$.
\end{lem}
\begin{proof}
The necessity is clear. Let us prove the sufficiency. Take arbitrarily $g\in\calL^1(\calM,\mathbf{P})$. We can prove by induction on $k$ that
$$
\int_{\calM}\|\Phi(\lambda^k s,s,y)g(y)\|\,d\mathbf{P}(y)\geq c^k\int_{\calM}\|g(y)\|\,d\mathbf{P}(y),\quad\forall s\in\R_{\geq\delta},\forall k\in\Z.
$$
Let $t\geq 1$, $s\in\R_{\geq\gamma}$, where $\gamma:=\max\{\delta,\theta+1\}$. Let $p=\max\left\{j\in\Z:t\lambda^{-j}\geq 1 \right\}$. Then we have $\lambda^p\leq t<\lambda^{p+1}$. By \eqref{poly-growth}, we estimate
\begin{eqnarray*}
\int_{\calM}\|\Phi(\lambda^{p+1}s,s,y)g(y)\|\,d\mathbf{P}(y)
&\leq& M\left(\dfrac{\lambda^{p+1}}{t}\right)^\omega\int_{\calM}\|\Phi(ts,s,y)g(y)\|\,d\mathbf{P}(y)\\
&\leq& M\lambda^\omega\int_{\calM}\|\Phi(ts,s,y)g(y)\|\,d\mathbf{P}(y),
\end{eqnarray*}
which gives
\begin{eqnarray*}
\int_{\calM}\|\Phi(ts,s,y)g(y)\|\,d\mathbf{P}(y)
&\geq& M^{-1}\lambda^{-\omega}\int_{\calM}\|\Phi(\lambda^{p+1}s,s,y)g(y)\|\,d\mathbf{P}(y)\\
&\geq& M^{-1}\lambda^{-\omega}c^{p+1}\int_{\calM}\|g(y)\|\,d\mathbf{P}(y)\\
&\geq& Kt^\alpha\int_{\calM}\|g(y)\|\,d\mathbf{P}(y),
\end{eqnarray*}
where $\alpha=\frac{\ln c}{\ln\lambda}$, $K=M^{-1}\lambda^{-\omega}$, and the last inequality uses the fact that $p+1>\frac{\ln t}{\ln\lambda}$.
\end{proof}

\begin{lem}\label{t-bdd}
Let $(\Phi,\zeta)$ be a stochastic skew-evolution semiflow, which is polynomially bounded in mean (that is \eqref{poly-growth} holds). Suppose that there exist constants $\delta,L>0$, an above unbounded sequence $\{t_n\}\in\calS(a,b)$, such that the inequality
$$
\int_{\calM}\|\Phi(t_ns,s,y)g(y)\|\,d\mathbf{P}(y)\geq L\int_{\calM}\|g(y)\|\,d\mathbf{P}(y)
$$
holds for every $n\in\Z_{\geq 1}$, $s\in\R_{\geq\delta}$, and every $g\in\calL^1(\calM,\mathbf{P})$. Then there exists a constant $K>0$ such that the inequality
$$
\int_{\calM}\|\Phi(ts,s,y)g(y)\|\,d\mathbf{P}(y)\geq K\int_{\calM}\|g(y)\|\,d\mathbf{P}(y)
$$
holds for every $t\geq 1$, $s\geq\max\{\delta,\theta+1\}$, and every $g\in\calL^1(\calM,\mathbf{P})$.
\end{lem}
\begin{proof}
Let $t\geq 1$, $g\in\calL^1(\calM,\mathbf{P})$, and $s\in\R_{\geq\gamma}$, where $\gamma:=\max\{\delta,\theta+1\}$. Since $\lim\limits_{n\to\infty}t_n=\infty$, we can find $p\in\Z$ satisfying
$$
t_{2d}\geq t,\quad\forall d\geq p.
$$
Setting $j=\min\{k\in\{1,\cdots,p\}:t\leq t_{2k}\}$. There are two possibilities of $j$.

\noindent - If $j=1$, then $t\leq t_2$, and so by \eqref{poly-growth}
$$
\int_{\calM}\|\Phi(t_2s,s,y)g(y)\|\,d\mathbf{P}(y)\leq Mt_2^\omega\int_{\calM}\|\Phi(ts,s,y)g(y)\|\,d\mathbf{P}(y),
$$
which gives, by the assumption, that
\begin{eqnarray*}
\int_{\calM}\|\Phi(ts,s,y)g(y)\|\,d\mathbf{P}(y)
&\geq& M^{-1}t_2^{-\omega}\int_{\calM}\|\Phi(t_2s,s,y)g(y)\|\,d\mathbf{P}(y)\\
&\geq& M^{-1}t_2^{-\omega}L\int_{\calM}\|g(y)\|\,d\mathbf{P}(y).
\end{eqnarray*}

\noindent - If $j\geq 2$, then $t\leq t_{2j}$ and $t>t_{2(j-1)}\geq t_j$. By \eqref{poly-growth}, we have
\begin{eqnarray*}
\int_{\calM}\|\Phi(t_{2j}s,s,y)g(y)\|\,d\mathbf{P}(y)
&\leq& M\left(\dfrac{t_{2j}}{t}\right)^\omega\int_{\calM}\|\Phi(ts,s,y)g(y)\|\,d\mathbf{P}(y)\\
&\leq& M\left(\dfrac{t_{2j}}{t_j}\right)^\omega\int_{\calM}\|\Phi(ts,s,y)g(y)\|\,d\mathbf{P}(y)\\
&\leq& Ma^\omega 2^{b\omega}\int_{\calM}\|\Phi(ts,s,y)g(y)\|\,d\mathbf{P}(y),
\end{eqnarray*}
and so
\begin{eqnarray*}
\int_{\calM}\|\Phi(ts,s,y)g(y)\|\,d\mathbf{P}(y)
&\geq& M^{-1}a^{-\omega} 2^{-b\omega}\int_{\calM}\|\Phi(t_{2j}s,s,y)g(y)\|\,d\mathbf{P}(y)\\
&\geq& M^{-1}a^{-\omega} 2^{-b\omega}L\int_{\calM}\|g(y)\|\,d\mathbf{P}(y).
\end{eqnarray*}
Thus, we can choose $K\leq \min\{M^{-1}a^{-\omega} 2^{-b\omega}L,M^{-1}t_2^{-\omega}L\}$.
\end{proof}

\subsection{Discrete-time version}
With all preparation in place, we now are ready to state and prove the main result, which can be regarded as a discrete variant of the Datko-type theorem.
\begin{thm}\label{datko-1}
Let $(\Phi,\zeta)$ be a stochastic skew-evolution semiflow, which is polynomially bounded in mean (that is \eqref{poly-growth} holds). Then $(\Phi,\zeta)$ is polynomially unstable in mean if and only if it is injective in the stochastic sense (that is \eqref{injective-defn} holds) and there exist $\delta>0$, a Banach sequence space $\calE\in\calH(\Z_{\geq 1})$, and $\{t_n\}\in\calS(a,b)$ such that
\begin{enumerate}
\item for every $(s,g)\in\R_{\geq\delta}\times (\calL^1(\calM,\mathbf{P})\setminus\{0\})$, the sequence
$$
\psi_{s,g}:\Z_{\geq 1}\to\R_{\geq 0},\quad\psi_{s,g}(j)=\left[\int_{\calM}\|\Phi(t_j s,s,y)g(y)\|\,d\mathbf{P}(y)\right]^{-1}
$$
belongs to $\calE$.
\item there exists $K>0$ such that
$$|\psi_{s,g}(\cdot)|_{\calE}\leq K\left[\int_{\calM}\|g(y)\|\,d\mathbf{P}(y)\right]^{-1},\quad\forall s\in\R_{\geq\delta},\forall g\in\calL^1(\calM,\mathbf{P})\setminus\{0\}.$$
\end{enumerate}
\end{thm}
\begin{proof}
$\bullet$ Necessity. It is immediate by taking $\calE=l^1_w(\Z_{\geq 1})$ and $t_n=n$.

$\bullet$ Sufficiency. Let $\gamma :=\max\{\delta,\theta+1\}$, $g\in\calL^1(\calM,\mathbf{P})$, and $s\in\R_{\geq\gamma}$. 

We prove by a contradiction that $\{t_n\}$ is above unbounded. Indeed, we assume that $\mathbb{T} :=\sup\limits\{t_n:n\in\Z_{\geq 1}\}<\infty$. For every $j\in\{1,\cdots,k\}$, by \eqref{poly-growth} we have
\begin{eqnarray*}
\int_{\calM}\|\Phi(t_js,s,y)g(y)\|\,d\mathbf{P}(y)
&\leq& M\mathbb{T}^\omega\int_{\calM}\|g(y)\|\,d\mathbf{P}(y)
\end{eqnarray*}
which gives
\begin{eqnarray*}
\psi_{s,g}(\cdot)\geq M^{-1}\mathbb{T}^{-\omega}\left[\int_{\calM}\|g(y)\|\,d\mathbf{P}(y)\right]^{-1}\calX_{\{1,\cdots,k\}},
\end{eqnarray*}
and hence, we can estimate
\begin{eqnarray*}
&& K\left[\int_{\calM}\|g(y)\|\,d\mathbf{P}(y)\right]^{-1}\geq |\psi_{s,g}(\cdot)|_{\calE}\\
&&\geq M^{-1}\mathbb{T}^{-\omega}\left[\int_{\calM}\|g(y)\|\,d\mathbf{P}(y)\right]^{-1}\cdot|\calX_{\{1,\cdots,k\}}|_{\calE}.
\end{eqnarray*}
By \eqref{condi-HN}, we let $k\to\infty$ in the last inequality to get the contradiction.

Also by condition \eqref{condi-HN}, we can find $p$ with
$$
|\calX_{\{m,\cdots,mp\}}|_{\calE}\geq 1,\quad\forall m\in\Z_{\geq 1}.
$$
Let $m\in\Z_{\geq 1}$. For every $j\in\{m,\cdots,mp\}$, by \eqref{poly-growth} we have
\begin{eqnarray*}
\int_{\calM}\|\Phi(t_js,s,y)g(y)\|\,d\mathbf{P}(y)
&\leq& M\left(\dfrac{t_j}{t_m}\right)^\omega\int_{\calM}\|\Phi(t_ms,s,y)g(y)\|\,d\mathbf{P}(y)\\
&\leq& M\left(\dfrac{t_{mp}}{t_m}\right)^\omega\int_{\calM}\|\Phi(t_ms,s,y)g(y)\|\,d\mathbf{P}(y)\\
&\leq& Ma^\omega p^{b\omega}\int_{\calM}\|\Phi(t_ms,s,y)g(y)\|\,d\mathbf{P}(y),
\end{eqnarray*}
which gives
$$
\psi_{s,g}(\cdot)\geq M^{-1}a^{-\omega} p^{-b\omega}\left[\int_{\calM}\|\Phi(t_ms,s,y)g(y)\|\,d\mathbf{P}(y)\right]^{-1}\calX_{\{m,\cdots,mp\}}.
$$
Hence,
\begin{eqnarray*}
&& K\left[\int_{\calM}\|g(y)\|\,d\mathbf{P}(y)\right]^{-1}\geq |\psi_{s,g}(\cdot)|_{\calE}\\
&&\geq M^{-1}a^{-\omega} p^{-b\omega}\left[\int_{\calM}\|\Phi(t_ms,s,y)g(y)\|\,d\mathbf{P}(y)\right]^{-1}|\calX_{\{m,\cdots,mp\}}|_{\calE}\\
&&\geq M^{-1}a^{-\omega} p^{-b\omega}\left[\int_{\calM}\|\Phi(t_ms,s,y)g(y)\|\,d\mathbf{P}(y)\right]^{-1}.
\end{eqnarray*}
Now we can use Lemma \ref{t-bdd} to see that there exists a constant $L>0$ such that
$$
\int_{\calM}\|\Phi(ts,s,y)g(y)\|\,d\mathbf{P}(y)\geq L\int_{\calM}\|g(y)\|\,d\mathbf{P}(y),\quad\forall t\geq 1,\forall s\geq\gamma.
$$

Let $\ell\in\Z_{\geq 1}$. For every $j\in\{1,\cdots,\ell\}$, we can estimate
$$
\int_{\calM}\|\Phi(t_{\ell}s,s,y)g(y)\|\,d\mathbf{P}(y)\geq L\int_{\calM}\|\Phi(t_j s,s,y)g(y)\|\,d\mathbf{P}(y),
$$
which is equivalent to the fact that
$$
\psi_{s,g}(\cdot)\geq L\left[\int_{\calM}\|\Phi(t_{\ell}s,s,y)g(y)\|\,d\mathbf{P}(y)\right]^{-1}\calX_{\{1,\cdots,\ell\}}.
$$
By the assumption, we estimate
\begin{eqnarray*}
&& K\left[\int_{\calM}\|g(y)\|\,d\mathbf{P}(y)\right]^{-1}\geq |\psi_{s,g}(\cdot)|_{\calE}\\
&&\geq L\left[\int_{\calM}\|\Phi(t_{\ell}s,s,y)g(y)\|\,d\mathbf{P}(y)\right]^{-1}\cdot |\calX_{\{1,\cdots,\ell\}}|_{\calE}.
\end{eqnarray*}
On the other hand, by \eqref{poly-growth},
\begin{eqnarray*}
&&\int_{\calM}\|\Phi(t_\ell s,s,y)g(y)\|\,d\mathbf{P}(y)\\
&&\leq M\left(\dfrac{t_\ell}{[t_\ell]}\right)^\omega\int_{\calM}\|\Phi([t_{\ell}]s,s,y)g(y)\|\,d\mathbf{P}(y)\\
&&\leq M2^\omega\int_{\calM}\|\Phi([t_{\ell}]s,s,y)g(y)\|\,d\mathbf{P}(y).
\end{eqnarray*}
Thus, we have
$$
\left[\int_{\calM}\|\Phi([t_\ell]s,s,y)g(y)\|\,d\mathbf{P}(y)\right]^{-1}\cdot |\calX_{\{1,\cdots,\ell\}}|_{\calE}\leq KL^{-1}M2^\omega\left[\int_{\calM}\|g(y)\|\,d\mathbf{P}(y)\right]^{-1}.
$$
By \eqref{condi-HN} we can choose $\ell$ with $|\calX_{\{1,\cdots,\ell\}}|_{\calE} \geq 2KL^{-1}M2^\omega$, and so by Lemma \ref{lemma-important} we get the desired result.
\end{proof}

The following result is a direct consequence of the theorem above.
\begin{cor}
Let $(\Phi,\zeta)$ be a stochastic skew-evolution semiflow, which is polynomially bounded in mean (that is \eqref{poly-growth} holds). Then $(\Phi,\zeta)$ is polynomially unstable in mean if and only if it is injective in the stochastic sense (that is \eqref{injective-defn} holds), and there exist positive constants $K,\delta$ such that the inequality
$$
\sum_{j=1}^\infty\dfrac{1}{j}\left[\int_{\calM}\|\Phi(js,s,y)g(y)\|\,d\mathbf{P}(y)\right]^{-1}\leq K\left[\int_{\calM}\|g(y)\|\,d\mathbf{P}(y)\right]^{-1}
$$
holds for every $s\in\R_{\geq\delta}$ and every $g\in\calL^1(\calM,\mathbf{P})\setminus\{0\}$.
\end{cor}

\subsection{Continuous-time version}
In this subsection, we give a continuous-time version of the Datko-type theorem by making use of Theorem \ref{datko-1}.
\begin{thm}\label{datko-2}
Let $(\Phi,\zeta)$ be a stochastic skew-evolution semiflow, which is polynomially bounded in mean (that is \eqref{poly-growth} holds). Then $(\Phi,\zeta)$ is polynomially unstable in mean if and only if it is injective in the stochastic sense (that is \eqref{injective-defn} holds), and there exist $\delta>0$, $\calQ\in\calH(\R_{\geq 1})$ such that
\begin{enumerate}
\item for every $(s,g)\in\R_{\geq\delta}\times (\calL^1(\calM,\mathbf{P})\setminus\{0\})$, the function $f_{s,g}:\R_{\geq 1}\to\R_{\geq 0}$
\begin{equation*}
f_{s,g}(t)=\left[\int_{\calM}\|\Phi(ts,s,y)g(y)\|\,d\mathbf{P}(y)\right]^{-1}
\end{equation*}
belongs to $\calQ$;
\item there exists $K>0$ such that
$$|f_{s,g}(\cdot)|_{\calQ}\leq K\left[\int_{\calM}\|g(y)\|\,d\mathbf{P}(y)\right]^{-1},\quad\forall s\in\R_{\geq\delta},\forall g\in\calL^1(\calM,\mathbf{P})\setminus\{0\}.$$
\end{enumerate}
\end{thm}
\begin{proof}
$\bullet$ Necessity. It is immediate by taking $\calQ=L_w^1(\R_{\geq 1})$.

$\bullet$ Sufficiency. For every $(s,g)\in\R_{\geq\delta}\times\calL^1(\calM,\mathbf{P})$, let us define the sequence $\psi_{s,g}:\R_{\geq 1}\to\R_{\geq 0}$ by setting
$$\psi_{s,g}(j)=\left[\int_{\calM}\|\Phi(t_j s,s,y)g(y)\|\,d\mathbf{P}(y)\right]^{-1},$$
where $t_j=j$.

Let $t\geq 1$. Then there exists $k\in\Z_{\geq 1}$ such that $t\in [k,k+1)$, and hence, $[t]=k$. Thus, by \eqref{poly-growth}, we have
\begin{eqnarray*}
\int_{\calM}\|\Phi(ts,s,y)g(y)\|\,d\mathbf{P}(y)
&\leq& M\left(\dfrac{t}{k}\right)^\omega\int_{\calM}\|\Phi(t_ks,s,y)g(y)\|\,d\mathbf{P}(y)\\
&\leq& M2^\omega\int_{\calM}\|\Phi(t_ks,s,y)g(y)\|\,d\mathbf{P}(y)
\end{eqnarray*}
which gives $M 2^\omega f_{s,g}\geq \sum\limits_{k\geq 1}\psi_{s,g}(k)\calX_{[k,k+1)}$, and so
\begin{eqnarray}\label{ineq-im}
M 2^\omega |f_{s,g}|_{\calQ}\geq |\sum\limits_{k\geq 1}\psi_{s,g}(k)\calX_{[k,k+1)}|_{\calQ}.
\end{eqnarray}
Hence by Remark \ref{rem-import-banach-func}, $S_{\calQ}\in\calH(\Z_{\geq 1})$. By \eqref{ineq-im} and Theorem \ref{datko-1}, we conclude that $(\Phi,\zeta)$ is polynomially unstable in mean.
\end{proof}

As a consequence, we obtain the following result.
\begin{cor}
Let $(\Phi,\zeta)$ be a stochastic skew-evolution semiflow, which is polynomially bounded in mean (that is \eqref{poly-growth} holds). Then $(\Phi,\zeta)$ is polynomially stable in mean if and only if it is injective in the stochastic sense (that is \eqref{injective-defn} holds), and there exist positive constants $K,\delta$ such that the inequality
$$
\int\limits_1^\infty\dfrac{1}{t}\left[\int_{\calM}\|\Phi(ts,s,y)g(y)\|\,d\mathbf{P}(y)\right]^{-1}\,dt\leq K\left[\int_{\calM}\|g(y)\|\,d\mathbf{P}(y)\right]^{-1}
$$
holds for every $s\in\R_{\geq\delta}$ and every $g\in\calL^1(\calM,\mathbf{P})\setminus\{0\}$.
\end{cor}


\bibliographystyle{plain}
\bibliography{refs}
\end{document}